\documentclass[reqno]{amsart}

\usepackage{amsmath,amsthm}
\usepackage{dsfont} 
\usepackage[english]{babel}
\usepackage{color}
\usepackage{amssymb}
\usepackage{mathrsfs}
\usepackage{graphicx}
\usepackage[font=footnotesize, labelfont=bf]{caption}
\usepackage{mathtools}
\usepackage[colorlinks=true, pdfstartview=FitV, linkcolor=blue, citecolor=blue, urlcolor=blue,pagebackref=false]{hyperref}
\usepackage{microtype}
\usepackage{enumitem}
\usepackage[normalem]{ulem}
\usepackage{soul}

\usepackage{float}

\newtheorem{theorem}{Theorem}[section]

\newtheorem{lemma}[theorem]{Lemma}
\newtheorem{corollary}[theorem]{Corollary}

\theoremstyle{definition}
\newtheorem{remark}[theorem]{Remark}
\newtheorem{example}[theorem]{Example}
\newtheorem{definition}[theorem]{Definition}

\newtheorem{assumption}{Assumption}
\numberwithin{equation}{section}

\definecolor{myblue}{RGB}{0,0,128}

\def\Xint#1{\mathchoice
	{\XXint\displaystyle\textstyle{#1}}%
	{\XXint\textstyle\scriptstyle{#1}}%
	{\XXint\scriptstyle\scriptscriptstyle{#1}}%
	{\XXint\scriptscriptstyle\scriptscriptstyle{#1}}%
	\!\int}
\def\XXint#1#2#3{{\setbox0=\hbox{$#1{#2#3}{\int}$ }
		\vcenter{\hbox{$#2#3$ }}\kern-.57\wd0}}

\def\dashint{\Xint-}

\newcommand{\dx}{\,\mathrm{d}x}
\newcommand{\dy}{\,\mathrm{d}y}

\newcommand{\e}{\varepsilon}
\newcommand{\dist}{{\rm{dist}}}

\newcommand{\w}{\omega}

\newcommand{\R}{\mathbb{R}}

\newcommand{\N}{\mathbb{N}}

\newcommand{\F}{\mathcal{F}}

\newcommand{\D}{\mathrm{D}}

\begin{document}
	
\author{Lukas Koch}
\address[Lukas Koch]{MPI Mathematics in the Sciences, Inselstra\ss e 22, 04177 Leipzig, Germany}
\email{kochl@mis.mpg.de}

\author{Matthias Ruf}
\address[Matthias Ruf]{Section de math\'ematiques, EPF Lausanne, Station 8, 1015 Lausanne, Switzerland}
\email{matthias.ruf@epfl.ch}

\author{Mathias Sch\"affner}
\address[Mathias Sch\"affner]{Institut f\"ur Mathematik, MLU Halle-Wittenberg, Theodor-Lieser-Stra\ss e 5, 06120 Halle (Saale), Germany}
\email{mathias.schaeffner@mathematik.uni-halle.de}

\title{On the Lavrentiev gap for convex, vectorial integral functionals}

\begin{abstract}	
	We prove the absence of a Lavrentiev gap for vectorial integral functionals of the form
	$$
	\mathscr F: g+W_0^{1,1}(\Omega)^m\to\mathbb [0,+\infty],\qquad \mathscr F(u)=\int_\Omega J(x,\D u)\dx,
	$$
	where the boundary datum $g:\Omega\subset \R^d\to\R^m$ is sufficiently regular, $\xi\mapsto J(x,\xi)$ is convex and lower semicontinuous, satisfies $p$-growth from below and suitable growth conditions from above. More precisely, if $p\leq d-1$, we assume $q$-growth from above with $q\leq \frac{(d-1)p}{d-1-p}$, while for $p>d-1$ or $p=1$ if $ d=2$, we require essentially no growth conditions from above and allow for unbounded integrands. Concerning the $x$-dependence, we impose a well-known local stability estimate that is redundant in the autonomous setting, but in the general non-autonomous case can further restrict the growth assumptions. 
\end{abstract}

\maketitle
{\small
	\noindent\keywords{\textbf{Keywords:} Lavrentiev-gap, convex vectorial integral functionals, non-autonomous integrands}
	
	\noindent\subjclass{\textbf{MSC 2020:} 49K40}
}	


\section{Introduction}
In this paper we consider integral functionals of the form
\begin{align}\label{eq:prob}
	\int_\Omega J(x,\D u)\dx,
\end{align}
where $\Omega\subset\R^d$ is a bounded, open set with Lipschitz boundary, $u\colon \Omega\to \R^m$ and $J\equiv J(x,\xi)\colon \Omega\times\R^{m\times d}\to [0,+\infty]$ is such that $J(x,\cdot)$ is convex and lower semicontinuous. For parts of our results we additionally enforce the growth assumption
\begin{align}\label{eq:pq}
	\lvert \xi\rvert^p\leq J(x,\xi)\leq C(\lvert \xi\rvert^q+1), \quad \text{for all } \xi\in \R^{m\times d}.
\end{align}

Since the work of Lavrentiev \cite{Lavrentiev1926}, it is known that for a general integral functional $\F\colon W^{1,p}(\Omega)^m\to \R$, it may happen that
\begin{align*}
	\inf_{u\in W^{1,p}(\Omega)^m} \F[u] <\inf_{u\in C^\infty(\Omega)^m}\F[u].
\end{align*}
The occurrence of this phenomenon, which is known as Lavrentiev phenomenon, is a serious obstruction to regularity theory and numerical approximation, when it occurs. In theories of nonlinear elasticity the Lavrentiev phenomenon is related to physical phenomena, such as cavitation, and hence of particular interest \cite{Foss2003}.

The theory of Lavrentiev's phenomenon for functionals of the form \eqref{eq:prob} with $J\equiv J(x,\D u)$, was developed in \cite{Zhikov1987}, \cite{Zhikov1993} and \cite{Zhikov1995}. In this paper, we adopt the viewpoint and terminology of \cite{Buttazo1992} and view the Lavrentiev phenomenon through the so-called Lavrentiev gap. 
Suppose $X$ is a topological space of weakly differentiable functions and $Y\subset X$. Introduce the sequentially lower semicontinuous (slsc) envelopes
\begin{align*}
	\overline \F_X = \sup\{\,\mathscr G\colon X\to[0,+\infty]: \mathscr G \text{ slsc }, \mathscr G\leq \F \text{ on } X\,\}\\
	\overline \F_{Y} = \sup\{\,\mathscr G\colon X\to[0,+\infty]: \mathscr G \text{ slsc }, \mathscr G\leq \F \text{ on } Y\,\}\nonumber.
\end{align*}
The Lavrentiev gap functional is then defined for $u\in X$ as
\begin{align*}
	L(u,X,Y)=\begin{cases}
		\overline \F_{Y}(u)-\overline \F_X(u)  &\text{ if } \overline \F_X(u)<+\infty,\\
		0 &\text{ else}.
	\end{cases}
\end{align*}
Note that $L(u,X,Y)\geq 0$ and that $L(u,X,Y)>0$ for some $u\in X$ when the Lavrentiev phenomenon occurs. However, in general, it can happen that $L(u,X,Y)>0$ for some $u\in X$, but the Lavrentiev phenomenon does not occur. There is an extensive literature on the Lavrentiev phenomenon and gap functional in this abstract set-up, an overview of which can be found in \cite{Buttazo1995}, \cite{Foss2001} to which we also refer for further references. Here, we only comment that in all situations we consider, it holds that $\overline \F_X[u] = \int_\Omega J(x,\D u)\dx$ and
\begin{align*}
	\overline \F_Y[u] = \inf \left\{\liminf_{j\to\infty} \int_\Omega J(x,\D u_j)\dx\colon (u_j)\subset Y, u_j\rightharpoonup u \text{ weakly in } X\right\}.
\end{align*}
Thus, the Lavrentiev phenomenon is related to approximation properties of function spaces and there is a wide literature available studying it from this angle, see \cite{Ahmida2018,Borowski2022a,Chlebicka2021} and references therein.

Integrands satisfying \eqref{eq:pq} were first studied in the seminal papers \cite{Marcellini1989,Marcellini1991}. There is by now an extensive literature regarding the regularity theory of the corresponding integral functionals. We refer to \cite{Mingione2006,Mingione2021} for an overview and further references.
Concerning the Lavrentiev gap and considering integrands $J\equiv J(x,\xi)$ satisfying $(p,q)$-growth \eqref{eq:pq}, in the local setting, $X=W^{1,p}_{\mathrm{loc}}(\Omega)^m$ and $Y=W^{1,q}_{\mathrm{loc}}(\Omega)^m$, the situation is somewhat well understood. In \cite{Esposito2004} it was shown that the Lavrentiev phenomenon can occur if
\begin{align*}
	p<d<d+\alpha<q,
\end{align*}
when $J(\cdot,\xi)$ is $\alpha$-H\"older-continuous. The dependence of the above restriction on $d$ was recently removed in \cite{Diening2020}. In light of the regularity theory obtained in \cite{Esposito2004}, this range of $(p,q)$ is sharp. In order to prove non-occurrence of the Lavrentiev phenomenon it seems unavoidable to impose further structure conditions than just H\"older-continuity of $J(\cdot,\xi)$. A variety of possible assumptions are known, see e.g. \cite{deFilippis2022,Esposito2019,Harjulehto2018}. The key to all assumptions is to ensure that $J(x,\xi)\sim J(y,\xi)$ if $x\sim y$ and $\lvert \xi\rvert$ is not too large.

Recently, there has been an increasing interest in proving lack of the Lavrentiev gap in the non-autonomous setting in the presence of boundary conditions. In other words, considering $X=g+W^{1,p}_0(\Omega)^m$ and $Y=g+W^{1,q}_0(\Omega)^m$ for a suitable boundary value $g$. \cite{Bulicek2022} considers functionals modeled on the double-phase functional $J\equiv \lvert \xi\rvert^p+a(x)\lvert \xi\rvert^q$. Then, for $g\in W^{1,q}(\Omega)^m$, the Lavrentiev gap vanishes if
\begin{align*}
	q<p+\alpha \max\left(1,\frac p d,\right).
\end{align*}
This approach has been extended to certain anisotropic functionals in \cite{Borowski2022}. In \cite{Borowski2023} the occurrence of the Lavrentiev phenomenon was studied for the double-phase functional under a family $S_\kappa$ of smoothness assumptions. In particular, it was shown that the Lavrentiev phenomenon cannot occur if
\begin{align*}
	q<p+\kappa.
\end{align*}
Moreover, it is shown that this range is sharp. $C^{0,\alpha}$ and $C^{1,\alpha}$-regularity of $J(\cdot,\xi)$ implies that $J$ lies in $S_\alpha$ and $S_{1+\alpha}$, respectively. For $\kappa\geq 1$, the spaces $S_\kappa$ do not necessarily contain all integrands $J$ with $C^{\lfloor\kappa\rfloor,\kappa-\lfloor \kappa\rfloor}$-continuous $x$-dependence. However, none of these results cover the general non-autonomous, vectorial setting: in \cite{Bulicek2022,Borowski2023} the integrands depend essentially only on the norm of the gradient, while in \cite{Borowski2022} the authors require an additive decomposition into integrands depending only on the gradient of a single component of the vector-valued map $u$.

In this work we focus specifically on the vectorial setting. Let us first describe our results for autonomous integrands, that is $J\equiv J(\xi)$. Here, it is known that the Lavrentiev gap cannot occur in the one-dimensional case $d=1$ \cite{Alberti1994}, nor in the scalar case $m=1$ \cite{Dearcangelis1999}. Recently, it has even been shown that in the scalar setting $m=1$, there is no gap for functionals of the form $\int_\Omega f(u,\D u)\dx$, assuming only convexity in the gradient variable \cite{Bousquet2022}.
In the vectorial setting $d,m>1$, when $J\equiv J(\xi)$, it is easy to see by mollification arguments that without imposing boundary conditions, the Lavrentiev gap cannot occur, at least on compactly contained subsets of $\Omega$, regardless of the values of $(p,q)$ in \eqref{eq:pq}. While it is possible that even under the presence of sufficiently smooth boundary conditions, the Lavrentiev phenomenon/gap cannot occur, this seems non-trivial to show and the situation is unclear. In fact, we are not aware of any result focusing on the autonomous, vectorial case. In this paper we show that if $p>d-1$  or $p=1$ if $d=2$ , the Lavrentiev gap functional vanishes without any upper growth assumption. In fact, we even admit infinite integrands, our only assumption being that $0\in \mathrm{int}(\mathrm{dom}\; J)$. In case $p\leq d-1$, we prove that the Lavrentiev gap vanishes if $q<\frac{(d-1)p}{d-1-p}$ ($q<+\infty$ if $p=d-1$). Thus, our results can be seen as a partial extension of the scalar results to the vectorial setting, covering all exponents in dimension two.

Turning to the non-autonomous case, we require in addition the following standard stability estimate: for every $M>1$ there exists $C_M>0$ and a non-negative function $\alpha_M\in L^1(\Omega)$ such that
\begin{equation}\label{intro:stability}
	J(x,\xi)\leq C_M\left(\text{ess}\inf_{y\in \overline{B_{\delta}(x)}\cap\Omega}J(y,\cdot)\right)^{**}(\xi)+\alpha_M(x)
\end{equation}
for all $0<\delta\ll 1$, all $x\in\Omega$ and all $|\xi|\leq M\delta^{-d/p}$, where $^{**}$ denotes the biconjugate function with respect to the $\xi$-variable. For a more detailed discussion of this assumption see Remark \ref{r.onassump} (iii). Here the novelty of our results is that in contrast to \cite{Bulicek2022,Borowski2022,Borowski2023} we do not impose any structure on the gradient dependence. This comes with the drawback that, considering for instance an integrand of the form $J(x,\xi)=V_0(\xi)+a(x)V_1(\xi)$, where $V_0,V_1$ are non-negative and convex, $0\leq a\in C^{0,\alpha}(\Omega)$, $V_0$ satisfies a two-sided $p$-growth condition, and $V_1(\xi)\leq C(|\xi|^q+1)$, the admissible range of exponents is restricted to
\begin{align*}
	q<p+\alpha \frac p d,
\end{align*}
since without structural assumptions we are not able to reduce the analysis to $L^{\infty}$-functions via truncation. However, especially in the fast-growth regime $p>d-1$, we can handle new types of integrands, see Example \ref{example}. In particular, we highlight that we do not require any form of $\Delta_2$-condition. 

Using classical arguments, our results both for the case $p\leq d-1$ (Theorem \ref{thm:main}) and the case $p>d-1$ or $p=1$ if $d=2$ (Theorem \ref{thm:2}) also allow us to treat non-convex integrands with convex growth, see Corollary \ref{c.nonconvex}.

Our techniques in the convex case rely on a by now standard approach. By pulling in the boundary data, we put ourselves in a situation where mollification arguments apply. In the autonomous situation, there is no interaction between the two scales, while in the non-autonomous set-up it is essential to use the same scale for both the boundary argument and the mollification. The novelty of our approach lies in a careful use of convexity (especially for unbounded integrands) in combination with the use of an ad-hoc partition of unity that is adapted to the function $u\in g+W^{1,p}_0(\Omega)^m$. It is at this point that the integrability of the map $u$ itself has to be compatible with the growth of the energy, which then forces a relation between the $W^{1,p}$-regularity of $u$ and the upper bounds on the energy. In particular, the assumptions could be drastically weakened if we either can assume that $u\in L^{\infty}(\Omega)^m$ or when $\Omega$ is star-shaped with respect to a ball (in this case no partition of unity is necessary).

The outline of our paper is as follows: In Section \ref{sec:main} we establish our notation, collect the assumptions and state our main results. We also give examples of the types of integrands we are able to consider there. The proofs are postponed to Section \ref{sec:proofs}, while in Appendix \ref{app:stronglystarshaped} we provide the necessary details to extend our results to locally star-shaped sets in the case of autonomous integrands. 

\section{Statement of the main result}\label{sec:main}
Let $d\geq 2$ and let $\Omega\subset\R^d$ be a bounded, open set with Lipschitz boundary and let $p\in [1,+\infty)$. We consider non-autonomous integral functionals of the form
\begin{equation*}
	F(u)=\int_\Omega J(x,\D u(x))\dx\in [0,+\infty],\qquad u\in W^{1,p}(\Omega)^m,
\end{equation*}
where the integrand $J:\Omega\times\R^{m\times d}\to [0,+\infty]$ has the following properties:
\begin{assumption}\label{ass1}
The function $J:\Omega\times\R^{m\times d}\to [0,+\infty]$ is jointly measurable and satisfies
\begin{enumerate}[label=(a\arabic*)]
	\item\label{(a1)} for a.e. $x\in \Omega$ the map $\xi\mapsto J(x,\xi)$ is convex and lower semicontinuous;
	\item\label{(a2)} for a.e. $x\in \Omega$ it holds that $|\xi|^p\leq J(x,\xi)$ for all $\xi\in\R^{m\times d}$;
	\item\label{(a3)} there exist $C,\theta>0$ and a non-negative function $\alpha \in L^1(\Omega)$ such that
	\begin{equation}\label{eq:growth_ass}
		\begin{cases}
			J(x,\xi)\leq C|\xi|^q+\alpha(x) &\mbox{if $p\leq d-1$  and $d\geq 3$,}
			\\
			\sup_{|\xi|\leq\theta }J(x,\xi)\in L^1(\Omega) &\mbox{if $p>d-1$ or $d=2$ ,}
		\end{cases}
	\end{equation}
	where in the first case the exponent $q>p$ satisfies the bound
	\begin{equation}
		\begin{cases} q\leq \frac{(d-1)p}{d-1-p} &\mbox{if $ p  <d-1$,}
			\\
			q< +\infty &\mbox{if $p=d-1$.}
		\end{cases}
	\end{equation}
	\item\label{(a4)} for every $M>1$ there exists $C_M>0$ and a non-negative function $\alpha_M\in L^1(\Omega)$ such that
	\begin{equation}\label{eq:stability}
		J(x,\xi)\leq C_M\left(\text{ess}\inf_{y\in \overline{B_{\delta}(x)}\cap\Omega}J(y,\cdot)\right)^{**}(\xi)+\alpha_M(x)
	\end{equation}
	for all $0<\delta\ll 1$, all $x\in\Omega$ and all $|\xi|\leq M\delta^{-d/p}$. Here $^{**}$ denotes the biconjugate function with respect to the $\xi$-variable.
\end{enumerate}
\end{assumption}
Let us comment on the assumptions.
\begin{remark}\label{r.onassump}
\begin{enumerate}[label=(\roman*)]
	\item Our assumptions yield new results especially in the autonomous setting, where Assumption \ref{(a4)} is redundant.
	\item The integrability assumption on the local supremum in \eqref{eq:growth_ass} is equivalent to the integrability of $J(x,\xi_i)$ for finitely many $\xi_1,\ldots\xi_n$ that generate a simplex with non-empty interior. This follows from the fact that convex functions attain their maximum in the corners of such a simplex.
	\item Assumption \ref{(a4)} is a slight variation of the version in \cite[p. 89]{Chlebicka2021}, where it was used to show modular density of $C_c^{\infty}$ in Sobolev-Musielak-Orlicz spaces. It is of very abstract nature and allows us to control the effect of convolution in the non-autonomous setting. In \cite[Lemma 6.2]{Bulicek2022} it is shown that the biconjugate can be omitted in the isotropic setting, i.e, when $J$ only depends on the modulus of $|\xi|$, and when $J$ satisfies a polynomial upper bound. On the other hand, in \cite[Remark 3.7.6]{Chlebicka2021} it is emphasized that in the general anisotropic case there is no hope to avoid the convex envelope. It can be omitted provided the infimum still defines a convex, lower semicontinuous function with respect to $\xi$ (see \cite[Theorem 4.92 (iii)]{FoLe}), which is for instance ensured by the property that the infimum is attained at a point $\widehat{y}\in \overline{B_{\delta}(x)}$ that does not depend on $\xi$. This condition was considered for instance in \cite{Esposito2019} to prove the absence of a local Lavrentiev gap for non-autonomous integrands with $(p,q)$-growth. We emphasize that in general \ref{(a4)} can impose further restrictions on the exponents in \ref{(a3)}. Some classes of integrands satisfying \ref{(a4)} are presented in Example \ref{example}.
\end{enumerate}
\end{remark}
We first state the theorem for $p\leq d-1$, where our assumptions cover the case of functionals with $(p,q)$-growth.
\begin{theorem}\label{thm:main}
Let $p\in [1,d-1]$. Assume $J:\Omega\times\R^{m\times d}\to [0,+\infty)$ satisfies Assumption \ref{ass1}. Let $g\in W^{1,p}(\Omega)^m$ be such that $\int_\Omega J(x,s\D g(x))\dx<+\infty$ for all $s\in\R$. Then for all $u\in g+W^{1,p}_0(\Omega)^m$ there exists a sequence $u_n\in g+C_c^{\infty}(\Omega)^m$ such that $u_n\to u$ in $W^{1,p}(\Omega)^m$ and 
\begin{equation*}
	\lim_{n\to +\infty}\int_\Omega J(x,\D u_n)=\int_\Omega J(x,\D u)\dx.
\end{equation*}
In particular, it holds that
\begin{equation*}
	\inf_{u\in g+W^{1,p}_0(\Omega)^m}\int_\Omega J(x,\D u)\dx =\inf_{u\in g+C_c^{\infty}(\Omega)^m}\int_\Omega J(x,\D u)\dx.
\end{equation*}
\end{theorem}
The next theorem concerns the cases $p>d-1$ and $p=1$ if $d=2$ with possibly unbounded integrands. Here we need to assume stronger regularity of the boundary datum. 
\begin{theorem}\label{thm:2}
	Let $p>d-1$ or $p=1$ if $d=2$. Assume $J:\Omega\times\R^{m\times d}\to [0,+\infty]$ satisfies Assumption \ref{ass1}. Let $g\in C_c^1(\R^d)^m$ satisfy $\int_\Omega J(x,s\D g)\dx<+\infty$ for some $s>1$. Then for all $u\in g+W^{1,p}_0(\Omega)^m$ there exists $u_n\in g+C_c^{\infty}(\Omega)^m$ such that $u_n\to u$ in $W^{1,p}(\Omega)^m$ and
	\begin{equation*}
		\lim_{n\to +\infty}\int_\Omega J(x,\D u_n)=\int_\Omega J(x,\D u)\dx.
	\end{equation*}
	In particular, it holds that
	\begin{equation*}
		\inf_{u\in g+W^{1,p}_0(\Omega)^m}\int_\Omega J(x,\D u)\dx =\inf_{u\in g+C^\infty_c(\Omega)^m}\int_\Omega J(x,\D u)\dx.
	\end{equation*}
\end{theorem}
Our third result concerns the case when in contrast to Theorem \ref{thm:2} the integrand $J$ is finite-valued, where the regularity of the boundary data can be weakened.
\begin{corollary}\label{c.locallybounded}
Let $p>d-1$ or $p=1$ if $d=2$. Assume that $J:\Omega\times\R^{m\times d}\to [0,+\infty)$ satisfies Assumption \ref{ass1}. Let $g\in W^{1,\infty}(\Omega)^m$. Then for all $u\in g+W^{1,p}_0(\Omega)^m$ there exists $u_n\in g+C_c^{\infty}(\Omega)^m$ such that $u_n\to u$ in $W^{1,p}(\Omega)^m$ and
\begin{equation*}
	\lim_{n\to +\infty}\int_\Omega J(x,\D u_n)=\int_\Omega J(x,\D u)\dx.
\end{equation*}
In particular, it holds that
\begin{equation*}
	\inf_{u\in g+W^{1,p}_0(\Omega)^m}\int_\Omega J(x,\D u)\dx =\inf_{u\in g+C^{\infty}_c(\Omega)^m}\int_\Omega J(x,\D u)\dx.
\end{equation*}
\end{corollary}
Our final result deals with non-convex integrands that have convex growth.
\begin{corollary}\label{c.nonconvex}
Let $G:\Omega\times\R^{m\times d}\to [0,+\infty]$ be jointly measurable and continuous in the second variable.	Assume that there exists a non-negative function $\alpha\in L^1(\Omega)$, a constant $C>1$ and a function $J:\Omega\times\R^{m\times d}\to [0,+\infty]$ as in either Theorem \ref{thm:main}, Theorem \ref{thm:2} or Corollary \ref{c.locallybounded} satisfying
\begin{equation*}
	\frac{1}{C}J(x,\xi)-\alpha(x)\leq G(x,\xi)\leq C J(x,\xi)+\alpha(x)\qquad\text{ for all }(x,\xi)\in\Omega\times\R^{m\times d}.
\end{equation*}
Then the conclusions of Theorem \ref{thm:main}, Theorem \ref{thm:2} or Corollary \ref{c.locallybounded} hold, respectively, with $J$ replaced by $G$. 
\end{corollary}
\begin{remark}\label{r.comments} Some comments are in order:
	\begin{enumerate}[label=(\roman*)] 
		\item The integrability assumption on $g$ in Theorem \ref{thm:main} is satisfied whenever $g\in W^{1,q}(\Omega)^m$.
		\item For unbounded integrands the integrability assumption on $g$ in Theorem \ref{thm:main} would not be satisfactory. The weaker assumption in Theorem \ref{thm:2} comes with the drawback that we have to require $C^1$-regularity of $g$. If we assume that $J(\cdot,s\D g)\in L^1(\Omega)$ for all $s\in\R$, then in terms of regularity $g\in W^{1,p}(\Omega)^m$ suffices since we can use the same argument as for Theorem \ref{thm:main}.
		\item In Theorem \ref{thm:2} the slightly stronger integrability of $\D g$ for some $s>1$ is redundant if we assume a suitable $\Delta_2$-condition.
	\end{enumerate}
	
\end{remark}
\begin{example}\label{example} Below we provide some examples of integrands covered by our results.
\begin{enumerate}
\item ($(p,q)$-double phase functionals). Consider $J(x,\xi)=V_0(\xi)+a(x)V_1(\xi)$ with $0\leq a\in C^{0,\alpha}(\Omega)$, $\alpha\in(0,1]$ and $V_0,V_1$ are nonnegative, convex and satisfy $|\xi|^p\leq V_0(\xi)\leq C(|\xi|^p+1)$, $0\leq V_1(\xi)\leq C(|\xi|^q+1)$. Standard computations yield that $J$ satisfies Assumption~\ref{ass1} provided

%
$$
\mbox{(i) $\frac1p-\frac1{d-1}\leq\frac1q$ (to ensure (a3)),}\qquad\mbox{(ii) $\frac{q}p\leq 1+\frac{\alpha}d$ (to ensure (a4)).}
$$
%
Note that condition (ii) implies (i). For this class of integrands, if $V_0,\; V_1$ are even, the conclusions of Theorem~\ref{thm:main} are already known, see e.g. \cite{Bulicek2022,Borowski2022,Borowski2023} and in view of counterexamples in \cite{Diening2020,Esposito2004} are optimal for $p\geq d$. Let us mention that in \cite{Bulicek2022} the absence of the Lavrentiev gap is proven under weaker assumptions on the exponents, namely $q\leq p+\alpha \max\{1,\frac{p}d\}$, provided $V_0,V_1$ are radial. The argument in \cite{Bulicek2022} relies on an additional truncation argument which seems not applicable in the vectorial setting without additional structure assumptions on the integrand. 



\item (double phase without $\Delta_2$). Let us briefly describe an extreme version of the double phase functional covered by Corollary~\ref{c.nonconvex}: Fix exponents $p>d-1$, $0<q\leq\tfrac{d-1}{d}<1$ and $\alpha>\frac{dq}{p-dq}>0$, and consider
$$
G(x,\xi)=|\xi|^p+a(x_1)\exp(|\xi|^q)\quad\mbox{with}\quad a(s)=\exp\left(-\frac1{s^\alpha}\right){\mathds 1}_{(0,+\infty)}(s)
$$
where $x=(x_1,x')\in\R^d$. Since $q<1$, the function $G$ is not globally convex. However, a direct calculation of the second derivative shows that $x\mapsto \exp(x^q)$ is convex on $[x_q,+\infty)$, where $x^q_q:=\frac{1-q}{q}$, while it is concave on $[0,x_q]$. Based on this observation, it is not difficult to check that at $x_*^q:=\frac{1}{q}>x_q^q$ the tangent line $t_q(s)=\exp(x^q_*)qx_*^{q-1}s$ at the graph of $x\mapsto \exp(x^q)$ always lies below the graph (due to concavity it suffices to verify the two points $0$ and $x_q$; at $x_q$ the claim follows from convexity). Thus, defining the function 
\begin{equation*}
	J(x,\xi)=\begin{cases}
		|\xi|^p+a(x_1)\exp(|\xi|^q) &\mbox{if $|\xi|>x_*$},
		\\
		|\xi|^p+a(x_1) t_q(|\xi|)&\mbox{if $|\xi|\leq x_*$},
	\end{cases}
\end{equation*}
we have that $J\leq G\leq J+C_*$ with $C_*=\exp(x_*^q)$ and $J$ satisfies (a1)--(a3) of Assumption~\ref{ass1}. We check condition (a4). Fix $\delta\in(0,\frac12]$ and $M>1$. We distinguish the cases $x_1\leq 2\delta^\frac1{\alpha+1}$ and $x_1>2\delta^\frac1{\alpha+1}$. For $x=(x_1,x')$ with $x_1\leq 2\delta^\frac1{\alpha+1}$, we have for all  $|\xi|\leq M\delta^{-\frac{d}p}$
$$
J(x,\xi)\leq G(x,\xi) \leq |\xi|^p+\exp(-2^{-\alpha}\delta^{-\frac\alpha{\alpha+1}})\exp(M^q\delta^{-\frac{dq}p})\leq |\xi|^p+\alpha_M
$$
with $\alpha_M:=\sup_{\delta\in(0,\frac12]}\exp(-2^{-\alpha}\delta^{-\frac\alpha{\alpha+1}}+M^q\delta^{-\frac{dq}p})<\infty$, where we use that $\frac\alpha{\alpha+1}>\frac{dq}p$. Hence, stability condition \eqref{eq:stability} is satisfied. For $x=(x_1,x')$ with $x_1>2\delta^\frac{1}{\alpha+1}$ it holds that
$$
\left(\text{ess}\inf_{y\in \overline{B_{\delta}(x)}\cap\Omega}J(y,\cdot)\right)^{**}(\xi)\geq J((x_1-\delta)e_1,\xi)
$$
and an elementary computation, using $x_1-\delta>\frac{x_1}2$ and $x_1^{\alpha+1}\geq 2^{\alpha+1}\delta$, yields
\begin{align*}
\frac{a(x_1)}{a(x_1-\delta)}=&\exp\left(\int_{x_1-\delta}^{x_1}\frac{s^{-\alpha-1}}{\alpha+1}\,\mathrm{d}s\right)\leq \exp\left(\frac{\delta}{(\alpha+1)(x_1-\delta)^{\alpha+1}}\right)\leq\exp\left(\frac{\delta2^{\alpha+1}}{(\alpha+1)x_1^{\alpha+1}}\right)\\
\leq&\exp\left(\frac{1}{(\alpha+1)}\right).
\end{align*}
Hence, for all $M>1$ stability condition \eqref{eq:stability} is satisfied (with $C_M=\exp((\alpha+1)^{-1})$) also in that case.

We remark that the integrand in this example is non-doubling. In particular, it is not covered by \cite{Bulicek2022,Borowski2023}. In addition, it is easy to check that also the anistropy condition of \cite[Remark 4.1]{Borowski2022} does not hold.
\item (fast anisotropic autonomous growth) We give an example of the type of anisotropic integrand to which Theorem \ref{thm:2} applies.  Consider with $a_{i,j},\,q_{i,j}\geq 1$, 
\begin{align*}
J(\xi) = \begin{cases} 
e^{\sum_{i=1}^m \sum_{j=d}^m\lvert a_{i,j}\xi_{ij}\rvert^{q_{i,j}}} \qquad& \text{ if } \xi_{11}>-1\\
+\infty &\text{ else}
\end{cases}
\end{align*}
Note that \ref{(a1)} and \ref{(a2)} are clearly satisfied for any $p>1$.  Since $0\in \mathrm{int}(\mathrm{dom}\; J)$, \ref{(a3)} also holds.

\item If $p>d-1$, we can add to any of the examples above a penalization term of the form
\begin{equation*}
	\Phi_K(\xi)=\begin{cases}
		0 &\mbox{if $\xi\in K$,}\\
		+\infty &\mbox{otherwise,}
	\end{cases}
\end{equation*}
where $K\subset\R^{m\times d}$ is a convex set with $0\in \text{int}(K)$. Such a term of course limits the possible boundary conditions.
\end{enumerate}
\end{example}
\section{Proofs}\label{sec:proofs}

\subsection{Construction of cut-off functions} In the following lemma we generalize \cite[Lemma~4.10]{RS23} and construct suitable cut-off functions $\eta$ which optimize the integrability of $\nabla\eta \otimes u_i$ for finitely many functions $u_i$, exploiting the exponent improvement for the Sobolev embedding in lower dimensions;  see \cite{BS21,BC16} for related results. These cut-off functions are the key ingredient to localize our constructions.
\begin{lemma}\label{L:optim}
Let $N\in\mathbb N$, $p\geq 1$ and $q>p$ satisfy
\begin{equation*}
	\begin{cases}
		q\leq \frac{(d-1)p}{d-1-p} &\mbox{if $p<d-1$,}
		\\
		q<+\infty &\mbox{if $p=d-1$,}
		\\
		q=+\infty &\mbox{if $p>d-1$ or $d=2$.}
	\end{cases}
\end{equation*}
Then there exists $C=C_{N,d,m,p,q}<+\infty$ such that the following is true: 

For any ball $B_R=B_R(x_0)$, any $u_1,\dots,u_N\in W^{1,p}(B_R)^m$ and $\delta\in(0,\frac12]$ there exists $\eta\in W_0^{1,\infty}(B_R)$ satisfying
\begin{equation}\label{L:optim:eta}
0\leq \eta\leq 1,\quad \eta=1\quad\mbox{in $B_{(1-\delta)R}$},\quad\|\nabla \eta\|_{L^\infty(B_R)}\leq \frac{2}{\delta R}
\end{equation}
and for all $i\in\{1,\ldots,N\}$
\begin{equation}\label{L:optim:claim}
\|\nabla \eta \otimes u_i\|_{L^q(B_R)}\leq \frac{CR^{\frac{d}{q}-\frac{d}{p}-1}}{\delta ^{1+\frac{1}{p}-\frac{1}{q}}}\left(R\biggl(\int_{B_R\setminus B_{(1-\delta) R}}|\D u_i|^p\dx\biggr)^\frac1p+\biggl(\int_{B_R\setminus B_{(1-\delta) R}}| u_i|^p\dx\biggr)^\frac1p\right)
\end{equation}
with the convention that $d/+\infty=1/+\infty=0$. 
\end{lemma}

\begin{proof}

The result for the case $p>d-1$ was shown in \cite[Lemma 4.10]{RS23}, so we assume that $p\in [1,d-1]$ and extend the argument to these exponents. Without loss of generality, we can suppose that $x_0=0$. Set $S_1:=\{x\in \R^d\,:\,|x|=1\}$.

\smallskip

{\bf Step 1} We prove the statement for $u_1,\dots,u_N\in C^1(B_R)^m$. For $i\in\{1,\dots,N\}$ and $C:=4N$, we set
\begin{equation}\label{def:U}
U_i:=\biggl\{r\in[(1-\delta) R,R]\,:\,\int_{S_1}|\D u_i(rz)|^p\,\mathrm{d}\mathcal H^{d-1}(z)\leq \frac{C}{\delta(1-\delta)^{d-1}R^d}\int_{B_R\setminus B_{(1-\delta) R}}|\D u_i|^p\dx\biggr\}.
\end{equation}
Fubini's Theorem and the definition of $U_i$ in the form
\begin{align*}
\int_{B_R \setminus B_{(1-\delta) R}}|\D u(x)|^p\dx=&\int_{(1-\delta) R}^Rr^{d-1}\int_{S_1}|\D u(r z)|^p\,\mathrm{d}\mathcal H^{d-1}(z)\,\mathrm{d}r\\
\geq&((1-\delta) R)^{d-1}\int_{((1-\delta) R,R)\setminus U_i}\int_{S_1}|\D u(rz)|^p\,\mathrm{d}\mathcal H^{d-1}(z)\,\mathrm{d}r\\
>&\frac{C(\delta R-|U_i|)}{\delta R}\int_{B_R \setminus B_{(1-\delta) R}}|\D u(x)|^p\dx
\end{align*}
imply $|U_i|\geq (1-\frac1C)\delta R$, or equivalently $|(1-\delta R,R)\setminus U_i|\leq \frac{\delta R}C$. Analogously,
\begin{equation}\label{def:U2}
V_i:=\biggl\{r\in[(1-\delta) R,R]\,:\,\int_{S_1}|u_i(rz)|^p\,\mathrm{d}\mathcal H^{d-1}(z)\leq \frac{C}{\delta(1-\delta)^{d-1}R^d}\int_{B_R\setminus B_{(1-\delta) R}}|u_i|^p\dx\biggr\}
\end{equation}
satisfies $|V_i|\geq (1-\frac1C)\delta R$, or equivalently $|(1-\delta R,R)\setminus V_i|\leq \frac{\delta R}C$. Setting $U:=\bigcap_{i=1}^NU_i\cap V_i$, from the choice $C=4N$ we obtain
\begin{equation}\label{bound:lowerU}
|U|\geq \delta R-\frac{2N}C\delta R=\frac{\delta R}2. 
\end{equation}
Next, we define $\eta\in W^{1,\infty}(B_R;[0,1])$ by
$$
\eta(x)=\tilde \eta(|x|),\quad\mbox{where}\quad \tilde \eta(r)=\begin{cases}1&\mbox{if $r\in(0,(1-\delta) R)$,}\\\displaystyle \frac{1}{|U|}\int_{r}^R\chi_{U}(s)\,ds&\mbox{if $r\in ((1-\delta) R,R)$.}\end{cases}
$$
By definition, we have that $0\leq \eta\leq1$, $\eta=1$ in $B_{(1-\delta) R}$, $\eta\in W_0^{1,\infty}(B_R)$ and for $x=rz$ with $r\in[0,R]$ and $z\in S_1$
\begin{equation}
|\nabla \eta(rz)|=\begin{cases}0&\mbox{if $r\notin U$,}\\\frac1{|U|}&\mbox{if $r\in U$.}
\end{cases}
\end{equation}
Hence, recalling \eqref{bound:lowerU}, the map $\eta$ satisfies all the properties in \eqref{L:optim:eta}. 

Next, we use that the exponent $q\in(p,+\infty]$ is such that $W^{1,p}(S_1)^m$ embeds continuously into $L^q(S_1)^m$,\footnote{This can be proven in details using charts for the sphere.} so that there exists $C<+\infty$ such that for all $v\in C^1(S_1)^m$ we have
\begin{equation*}
\|v\|_{L^q(S_1)}\leq C \|D_\tau v\|_{L^p(S_1)}+C\|v\|_{L^p(S_1)}
\end{equation*} 
%
where $\D_\tau$ denotes the tangential derivative. Applying the above estimate to $v_r\in C^1(S_1)^m$ defined by $v_r(z):=u(r z)$ for all $z\in S^1$ with $u\in C_1(B_R)^m$, we obtain with the chain rule
%
\begin{align}
\|u(r\cdot)\|_{L^q(S_1)}\leq C r \|\D u(r\cdot) \|_{L^p(S_1)}+C\|u(r\cdot)\|_{L^p(S_1)}.\label{compact:embedding1}
\end{align} 
Hence, via the change of coordinates $x=rz$ with $r=|x|$ and $z=\frac{x}{|x|}$ we deduce that 
\begin{align*}
\|\nabla \eta\otimes u_i\|_{L^q(B_R)}\leq&|U|^\frac1qR^{\frac{d-1}q}\|\nabla \eta\|_{L^\infty(B_R)}\sup_{r\in U}\|u_i(r\cdot)\|_{L^q(S_1)}\\
\leq&\frac{CR^{\frac{d-1}{q}}}{|U|^{1-\tfrac{1}{q}}} \sup_{r\in U}\left(r\|\D u_i(r\cdot) \|_{L^p(S_1)}+\|u_i(r\cdot)\|_{L^p(S_1)}\right)
\\
\leq& \frac{CR^{\frac{d}{q}-\frac{d}{p}-1}}{\delta ^{1+\frac{1}{p}-\frac{1}{q}}}\left(R \|\D u_i\|_{L^p(B_R\setminus B_{(1-\delta)R})}+\|u_i\|_{L^p(B_R\setminus B_{(1-\delta)R})}\right)
\end{align*}
where we use the definition of $U$, \eqref{bound:lowerU} and $1-\delta\geq \frac12$ in the last inequality. This proves the claim for $C^1$-functions.

\smallskip

{\bf Step 2} Conclusion. Consider $u_1,\dots,u_N\in W^{1,p}(B_R)^m$. By standard density results, we find $(u_{i,j})_j\subset C^\infty(B_R)^m$ such that $u_{i,j}\to u_i$ in $W^{1,p}(B_R)^m$. By Step~1, we find for every $j\in\mathbb N$ a cut-off function $\eta_j\in W_0^{1,\infty}(B_R)$ satisfying 
\begin{align}\label{L:optim:eta1}
&0\leq \eta_j\leq 1,\quad \eta_j=1\quad\mbox{in $B_{(1-\delta)R}$},\quad\|\nabla \eta_j\|_{L^\infty(B_R)}\leq \frac{2}{\delta R},\\
&\|\nabla \eta_j \otimes u_{i,j}\|_{L^q(B_R)}\leq\frac{CR^{\frac{d}{q}-\frac{d}{p}-1}}{\delta ^{1+\frac{1}{p}-\frac{1}{q}}}\left(R \biggl(\int_{B_R\setminus B_{(1-\delta)R}}|\D u_{i,j}|^p\dx\biggr)^\frac1p+\biggl(\int_{B_R\setminus B_{(1-\delta) R}}| u_{i,j}|^p\dx\biggr)^\frac1p\right).\label{L:optim:claim1}
\end{align}
In view of the bounds in \eqref{L:optim:eta1} and the Banach-Alaoglu Theorem, there exists $\eta\in W_0^{1,\infty}(B_R)$ such that up to subsequences (not relabeled) $\eta_j\stackrel{\star}{\rightharpoonup}\eta$ in $W^{1,\infty}(B_R)$. Moreover, $\eta$ also satisfies the bounds in \eqref{L:optim:eta}. Since $\nabla \eta_j\stackrel{\star}\rightharpoonup \nabla \eta$ weakly$^*$ in $L^\infty(B_R)^d$ and $u_{i,j}\to u_i$ (strongly) in $L^p(B_R)^m$, we deduce that $\nabla \eta_j\otimes u_{i,j}$ converges weakly in $L^p(B_R)^{m\times d}$ to $\nabla \eta\otimes u_i$ and by the boundedness of the right-hand side in \eqref{L:optim:claim1} also weakly in $L^{q}(B_R)^{m\times d}$ (weakly$^*$ if $q=+\infty$). Hence the claimed estimate \eqref{L:optim:claim} follows from \eqref{L:optim:claim1} and the weak or weak$^*$ lower-semicontinuity of the norm.

\end{proof}

\subsection{Approximation results}
In this section we show the approximation result that essentially will prove our main theorems. It heavily relies on Lemma~\ref{L:optim}. In Lemma~\ref{l.compactsupport} below we show the approximation claim for $C^1$-boundary conditions as in Theorem \ref{thm:2}. In the proof of Theorem \ref{thm:main} we will apply this lemma for zero boundary conditions. We start with the following elementary lemma. 
\begin{lemma}\label{l.interiorpoint}
	Let $J:\Omega\times\R^{m\times d}\to [0,+\infty]$ be jointly measurable and convex in the second variable. Assume that there exists $\theta>0$ such that $\sup_{|\xi|\leq\theta }J(x,\xi))\in L^1(\Omega)$. Let $g\in W^{1,p}(\Omega )^m$ be such that $\int_\Omega J(x,s\D g)\dx<+\infty$ for some $s>1$. Then there exists $r>0$ such that $\sup_{|\xi|\leq r}J(\cdot,\D g+\xi)\in L^1(\Omega)$. In particular, the map $\xi\mapsto J(x,\xi)$ is continuous at $\xi=\D g(x)$ for a.e. $x\in \Omega$.
\end{lemma}
\begin{proof}
	Let $\xi\in\R^{m\times d}$ be such that $|\xi|\leq \frac{s-1} {s}\theta$, where $\theta>0$ and $s>1$ are given by the assumptions. It follows from convexity that
	\begin{equation*}
		J(x,\D g+\xi)\leq \frac{1}{s}J(x,s\D g)+\frac{s-1}{s}J\left(x,\frac{s}{s-1}\xi\right)\leq J(x,s\D g)+\sup_{|\eta|\leq\theta}J(x,\eta).
	\end{equation*}
	As $J$ is non-negative, this implies the claimed integrability provided that $r\leq \frac{s-1}{s}\theta$. The second statement follows from the continuity of convex functions on the interior of their domain (see, e.g., \cite[Theorem 4.36]{FoLe}).
\end{proof}
\begin{lemma}\label{l.compactsupport}
Let $J:\Omega\times\R^{m\times d}\to [0,+\infty]$ satisfy Assumption \ref{ass1} and let $g\in C_c^1(\R^d)^m$. If for some $s>1$ it holds that $\int_\Omega J(x,s\D g)\dx<+\infty$, then for any $u\in g+W^{1,p}_0(\Omega)^m$ there exists a sequence $u_n\in g+C_c^{\infty}(\Omega)^m$ such that $u_n\to u$ strongly in $W^{1,p}(\Omega)^m$ and
\begin{equation*}
	\lim_{n\to +\infty}\int_\Omega J(x,\D u_n(x))\dx=\int_\Omega J(x,\D u(x))\dx.
\end{equation*}
\end{lemma}
\begin{proof}
	We divide the proof into several steps.
	\\
	{\bf Step 1:} Preliminary considerations
	
	It suffices to consider the case when $\int_\Omega J(x,\D u)\dx<+\infty$ since otherwise the statement follows from Fatou's lemma (recall that $J$ is lower semicontinuous in the second variable) and the density of $C_c^{\infty}(\Omega)^m$ in $W^{1,p}_0(\Omega)^m$. Moreover, writing $u=g+v$ with $v\in W_0^{1,p}(\Omega)^m$, it follows that $u_t:=g+tv$ converges to $u$ in $W^{1,p}(\Omega)^m$ when $t\uparrow 1$ and moreover
	\begin{equation*}
		\int_\Omega J(x,\D g+t\D v)\dx\leq t\int_\Omega J(x,\D g+\D v)\dx+(1-t)\int_\Omega J(x,\D g)\dx\overset{t\uparrow 1}{\to}\int_\Omega J(x,\D u)\dx,
	\end{equation*}
	so again due to Fatou's lemma and a diagonal argument it suffices to show the claim for functions $u=g+v\in g+ W^{1,p}_0(\Omega)^m$ such that $\int_\Omega J(x,\D g+s_0\D v)\dx<+\infty$ for some $s_0>1$. As we explain now, this entails that $J(x,\cdot)$ is continuous in $\D g(x)+\D v(x)$ for a.e. $x\in \Omega$. To this end, we show that $\D g(x)+\D v(x)$ belongs to the interior of $\text{dom}(J(x,\cdot))$. Given $s_0$ as above, let $\delta>0$ be such that $1-\frac{1}{s_0}>\delta$ and let $\xi\in\R^{m\times d}$ be such that $|\xi|\leq\delta r$, where $r>0$ is given by Lemma \ref{l.interiorpoint}. Then by convexity we have that
	\begin{equation*}
		J(x,\D g+\D v+\xi)\leq \frac{1}{s_0}J(x,\D g+s_0\D v)+\left(1-\frac{1}{s_0}-\delta\right)J(x,\D g)+\delta J(x,\D g+\delta^{-1}\xi)<+\infty
	\end{equation*}  
	for a.e. $x\in \Omega$, where we used Lemma \ref{l.interiorpoint} to control the third function on the RHS. The continuity of $\xi\mapsto J(x,\xi)$ in $\xi=\D g(x)+\D v(x)$ follows from the continuity of convex functions on the interior of their domain.
	\\
	{\bf Step 2:} Moving the support of $u-g$ inside $\Omega$
	
	Now fix $u\in g+W_0^{1,p}(\Omega)^m$ with the properties above. First extend $v=u-g$ to be $0$ outside $\Omega$, so that $v\in W^{1,p}(\R^d)^m$. For every $x\in \Omega$ we consider a ball $B_{r_x}(x)\subset\subset \Omega$, while for $x\in\partial \Omega$ the Lipschitz regularity of $\partial \Omega$ implies that (up to an Euclidean motion) there exists a cylinder $C_{x}= B^{d-1}_{r'_x}(0)\times (-h_x,h_x)$ with $x\in C_x$ and 
	\begin{equation}\label{eq:boundarydescription}
		\Omega\cap C_{x}=\{(y',y_d)\in C_x:\,y_d<\psi_{x}(y')\}
	\end{equation}
	for some Lipschitz-function $\psi_{x}:B^{d-1}_{r'_x}(0)\to(-h_x,h_x)$. Up to reducing $ r'_{x}$, we may assume that 
	\begin{equation}\label{eq:nottouching}
		\psi_x(B^{d-1}_{r'_x}(0))\subset\subset (-h_x,h_x).
	\end{equation}
	Choose then $r_x<\min\{r'_x,h_x\}$ such that $B_{r_x}(x)\subset\subset C_x$ and such that the Lipschitz-constant $L_x$ of $\psi_x$ satisfies
	\begin{equation}\label{eq:rx_small}
		0<2L_xr_x\leq h_x+\inf_{|y'|\leq r'_x}\psi_x(y'),
	\end{equation}
	which is possible due to \eqref{eq:nottouching}. Due to the compactness of $\overline{\Omega}$, we find a finite family of above balls $B_i=B_{r_{x_i}}(x_i)$ ($1\leq i\leq N$) that cover $\overline{\Omega}$. These balls will be fixed throughout the rest of the proof, so we omit the dependence on the radii $r_{x_i}$ or the number $N$ of certain quantities. For interior points $x_i\in \Omega$, we define $z_i=x_i$, while for points $x_i\in\partial \Omega$ we choose $z_i\in\R^d$ such that in the local coordinates we have $z_i=(0,-h_{x_i})$ (i.e., at the bottom of the local graph representation). Now let $0<\rho_k<1$ be such that $\lim_k\rho_k=1$ and for any $1\leq i\leq N$ we define the affine map $T_{k,i}$ and its inverse $T_{k,i}^{-1}$ by
	\begin{equation}\label{eq:defT_ki}
		T_{k,i}x=z_i+\rho_k(x-z_i),\qquad T_{k,i}^{-1}y=z_i+\frac{1}{\rho_k}(y-z_i),
	\end{equation}
where the purpose of the latter is to move the support of $v$ locally inside $\Omega$.
	Having this in mind, we define the functions
	\begin{equation*}
		v_{k,i}=\rho_kv\circ T_{k,i}^{-1},\qquad g_{k,i}=\rho_k g\circ T_{k,i}^{-1}.
	\end{equation*}
	Note that both functions are well-defined since $v$ and $g$ are defined on $\R^d$.
	Since $\rho_k\to 1$ and $g\in C_c^1(\R^d)^m$, via a density argument with respect to $v$ one shows that for all $1\leq i\leq N$ it holds that 
	\begin{equation}\label{eq:conv_ki}
		v_{k,i}\to v \text{ in }W^{1,p}(\R^d)^m,\qquad g_{k,i}\to g \text{ in }W^{1,\infty}(\R^d)^m\text{ as }k\to +\infty.
	\end{equation}
	Next, let $(\theta_{i})_{i=0}^N$ be a smooth partition of unity subordinated to the cover $\{\R^d\setminus\overline{\Omega},(B_{i})_{i=1}^N\}$ of $\R^d$ (note that we work on the 'manifold' $\R^d$ and therefore ${\rm supp}(\theta_{i})$ is compactly contained in $B_{i}$ and $\theta_{0}$ vanishes on a neighborhood of $\overline{\Omega}$). We build an ad hoc Lipschitz partition of unity as follows: choose $\delta_0>0$ such that for each $1\leq i\leq N$ we have ${\rm supp}(\theta_i)\subset\subset B_{(1-\delta_0)r_{x_i}}(x_i)$ and then use Lemma \ref{L:optim} (with $\delta=\delta_0$ that is fixed for the rest of the proof) for the finite family of functions $\{v_{k,j}-v\}_{j=1}^N\subset W^{1,p}(B_i)^m$ to obtain $\eta_{k,i}\in W_0^{1,\infty}(B_i)$ such that $0\leq\eta_{k,i}\leq 1$,
	\begin{equation*}
		\eta_{k,i}=1\quad\text{ on }{\rm supp}(\theta_i),\qquad\qquad \|\nabla\eta_{k,i}\|_{L^{\infty}(B_i)}\leq C
	\end{equation*}
	and for all $1\leq j\leq N$
	\begin{equation}\label{eq:productbound}
		\|\nabla\eta_{k,i}\otimes (v_{k,j}-v)\|_{L^{q}(B_i)}\leq C\|v_{k,j}-v\|_{W^{1,p}(B_i)}\overset{k\to +\infty}{\longrightarrow}0\qquad (q=\infty\text{ if }p>d-1).
	\end{equation}
	We define the Lipschitz partition of unity as follows: for $k\in\N$ we set
	\begin{equation*} \varphi_{k,0}=\frac{\theta_0}{\theta_0+\sum_{j=1}^N\eta_{k,j}},\quad \varphi_{k,i}=\frac{\eta_{k,i}}{\theta_0+\sum_{j=1}^N\eta_{k,j}}.
	\end{equation*}
	Here the denominator is always $\geq 1$ since $\eta_{k,j}=1$ on ${\rm supp}(\theta_j)$. Therefore the gradient of $\varphi_{k,i}$ ($1\leq i\leq N$) satisfies for any vector $a\in\R^m$ the pointwise estimate
	\begin{equation}\label{eq:grad_control}
		|\nabla\varphi_{k,i}\otimes a|\leq\frac{\left|\nabla\eta_{k,i}\otimes a\right|}{\theta_0+\sum_{j=1}^N\eta_{k,j}}+\frac{\eta_{k,i}\left|\nabla\theta_0\otimes a+\sum_{j=1}^N\nabla\eta_{k,j} \otimes a\right|}{\left(\theta_0+\sum_{j=1}^N\eta_{k,j}\right)^2}\leq |\nabla\theta_0\otimes a|+2\sum_{j=1}^N|\nabla\eta_{k,j}(x)\otimes a|.
	\end{equation}
	Set
	\begin{equation*}
		v_{k}(x)=\sum_{i=1}^{N}\varphi_{k,i}(x)v_{k,i}(x).
	\end{equation*}
	As a next step we aim to regularize $v_k$ via convolution at the scale $|\frac{1}{\rho_k}-1|\ll 1$. In order to ensure that the regularized function has compact support, we need to quantitatively control the support of $v_k$: let $ d_k \ll 1$ and consider $x\in \Omega$ such that $\dist(x,\partial \Omega)\leq d_k $. We show that for a suitably small constant $c_0>0$ the choice $d_k =c_0|\frac{1}{\rho_k}-1|$ implies that $v_k(x)=0$. To this end, consider a ball $B_i$ with $x\in B_i$. We argue that $\varphi_{k,i}(x)v_{k,i}(x)=0$ for such $i$, which then implies that $v_k(x)=0$. Since there are only finitely many sets in the covering and 'interior' balls are compactly contained in $\Omega$, for $ d_k $ small enough we know from our construction of the covering and \eqref{eq:boundarydescription} that, up to a Euclidean motion, with the corresponding cylinders $C_i=C_{x_i}$ and radii $r_i'=r_{x_i}'$ we can write
	\begin{equation*}
		\Omega\cap C_i=\{(y',y_d)\in C_i:\,y_d<\psi_i(y')\}
	\end{equation*}
	for a Lipschitz function $\psi_i:B_{r'_{i}}^{d-1}(0)\to (-h_i,h_i)$. Since ${\rm supp}(\varphi_{k,i})={\rm supp}(\eta_{k,i})\subset \overline{B_i}\subset\subset C_i$, there exists $0<\theta<\min_i r_i$ such that $\varphi_{k,i}(x)=0$ whenever $\dist(x,\partial C_i)\leq\theta$. Hence we can assume that $\dist(x,\partial C_i)>\theta$. We will show that $z_i+\frac{1}{\rho_k}(x-z_i)\in C_i\setminus \overline{\Omega}$, which then implies that $v_{k,i}(x)=0$. Since $x\in C_i$ and $\dist(x,\partial C_i)>\theta$ and the $z_i$ are fixed, it follows that for $\rho_k$ sufficiently close to $1$ we have $z_i+\frac{1}{\rho_k}(x-z_i)\in C_i$. Hence, in order to show that this point does not belong to $\overline{\Omega}$, it suffices to show that in the local coordinates (where $z_i=(0,-h_i)$) we have 
	\begin{equation}\label{eq:notinD}
		\psi_i\left(\frac{1}{\rho_k}x'\right)< -h_i+\frac{1}{\rho_k}(x_d+h_i).  
	\end{equation}
	Let $d_x\in\partial \Omega$ be such that $|x-d_x|=\dist(x,\partial \Omega)$. Note that $d_x\in C_i$ (otherwise the line from $x$ to $d_x$ intersects $\partial C_i$ at a distance less than $ d_k \ll\theta$), so that we can write $d_x=(y',\psi_i(y'))$ for some $y'\in B_{r'_i}^{d-1}(0)$. To show \eqref{eq:notinD}, let $L_i$ be the Lipschitz constant of $\psi_i$. Then we can estimate
	\begin{align*}
		\psi_i\left(\frac{1}{\rho_k}x'\right)-\left(\frac{1}{\rho_k}-1\right)h_i-\frac{1}{\rho_k}x_d&\leq \underbrace{\psi_i\left(y'\right)-x_d}_{\leq |d_x-x|\leq d_k }+L_i\left|y'-\frac{1}{\rho_k}x'\right|-\left(\frac{1}{\rho_k}-1\right)h_i-\left(\frac{1}{\rho_k}-1\right)x_d
		\\
		&\leq  d_k  +L_i\left(\frac{1}{\rho_k}-1\right)|x'|+L_i\underbrace{|y'-x'|}_{\leq |d_x-x|\leq  d_k }-\left(\frac{1}{\rho_k}-1\right)(h_i+x_d)
		\\
		&\leq (L_i+1) d_k +\left(\frac{1}{\rho_k}-1\right)(L_i|x'|-h_i-x_d)
		\\
		&\leq (L_i+2) d_k +\left(\frac{1}{\rho_k}-1\right)(L_i|x'|-h_i-\psi(y'))
		\\
		&\leq (L_i+2) d_k +\frac{1}{2}\left(\frac{1}{\rho_k}-1\right)(\underbrace{-h_i-\inf_{|y'|\leq r'_i}\psi(y')}_{=:\kappa_i\overset{\eqref{eq:nottouching}}{<}0}),
	\end{align*}
	where we used \eqref{eq:rx_small} in the last estimate. Hence we find $c_0>$ such that the choice $ d_k  =c_0\left(\frac{1}{\rho_k}-1\right)$ turns the right-hand side negative for all $1\leq i\leq N$. We thus proved that 
	\begin{equation}\label{eq:definescale}
		v_k=0\qquad\text{ on }\dist(\cdot,\partial \Omega)\leq c_0 \left|\frac{1}{\rho_k}-1\right|=:4\e_k.
	\end{equation}
	{\bf Step 3: Convolution and energy estimates}
	
	Denoting by $\phi_k$ a family of standard mollifiers supported in $B_{\e_k}(0)$, \eqref{eq:definescale} implies that we have $\phi_k\star v_k\in C_c^{\infty}(\Omega)^m$. We next show that $v_k\to v$ in $W^{1,p}(\Omega)^m$, which then also implies that $\phi_k\star v_k\to v$ in $W^{1,p}(\Omega)^m$. Indeed, due to Young's convolution inequality (recall that $v$ was extended by $0$ to $\R^d$, so the expression below makes sense) we have that
	\begin{equation*}
		\|\phi_k\star v_k-v\|_{L^p(\R^d)}\leq \|\phi_k\star v_k-\phi_k\star v\|_{L^p(\R^d)}+\|\phi_k\star v-v\|_{L^p(\R^d)}\leq \|v_k-v\|_{L^p(\R^d)}+\|\phi_k\star v-v\|_{L^p(\R^d)}\to 0
	\end{equation*}
	and the same argument shows the convergence of the gradients. By convexity, we have that $v_k\to v$ in $L^p(\Omega)^m$ due to \eqref{eq:conv_ki}, while for the gradients we first observe that on $\Omega$ (where $\varphi_{k,0}$ vanishes) the gradient of $v_k$ can be expressed as
	\begin{equation*}
		\D v_k=\sum_{i=1}^N\varphi_{k,i}\D v_{k,i}+\sum_{i=1}^N\nabla\varphi_{k,i}\otimes(v_{k,i}-v).
	\end{equation*}
	Hence the estimate \eqref{eq:grad_control} and the uniform gradient bound for $\eta_{k,i}$ imply that
	\begin{align*}
		\int_{\Omega}|\D v_{k}-\D v|^p\dx&\leq c(p)\int_{\Omega}\left|\sum_{i=1}^N\varphi_{k,i}(\D v_{k,i}-\D v)\right|^p\dx+c(p)\int_{\Omega}\left|\sum_{i=1}^N\nabla \varphi_{k,i}\otimes(v_{k,i}-v)\right|^p\dx
		\\
		&\leq C\sum_{i=1}^N\int_{\Omega}|v_{k,i}-v|^p+|\D v_{k,i}-\D v|^p\dx\to 0\quad\text{ as $k\to+\infty$ }. 
	\end{align*} 
	Now we establish the energy estimate. Given $t\in (0,1)$, we consider the sequence $u_{k,t}:=g+t\phi_k\star v_k\in g+C_c^{\infty}(\Omega)^m$. Then $u_{k,t}\to g+tv$ in $W^{1,p}(\Omega)^m$ as $k\to +\infty$. We claim that
	\begin{equation}\label{eq:energyclaim}
		\limsup_{t\uparrow 1}\limsup_{k\to +\infty}\int_\Omega J(x,\D u_{k,t})\dx\leq \int_\Omega J(x,\D u)\dx.
	\end{equation}
	To this end, we write the gradient of $u_{k,t}$ as
	\begin{align*}
		\D u_{k,t}&=t\sum_{i=1}^N \phi_k\star (\varphi_{k,i}(\D v_{k,i}+\D g_{k,i}))
		\\
		&\quad+\frac{(1-t)}{2}\left(\D g+\frac{2t}{1-t}\sum_{i=1}^N\phi_k\star(\varphi_{k,i}(\D g-\D g_{k,i})+\nabla\varphi_{k,i}\otimes(v_{k,i}-v))\right)
		\\
		&\quad +\frac{(1-t)}{2}\left(\D g+\frac{2t}{1-t}\sum_{i=1}^N\left(\varphi_{k,i}\D g-\phi_k\star(\varphi_{k,i}\D g)\right)\right).
	\end{align*}
	Thus, using convexity of $J(x,\cdot)$, we have that
	\begin{align*}
		\int_\Omega J(x,\D u_{k,t})\dx\leq \int_\Omega tJ\left(x,\sum_{i=1}^N \phi_k\star (\varphi_{k,i}(\D v_{k,i}+\D g_{k,i}))\right)\dx+\frac{1-t}{2}\int_\Omega J^t_{k,1}(x)+J^t_{k,2}(x)\dx
	\end{align*}
	with the error terms
	\begin{align*}
		J^t_{k,1}(x)&:=J\left(x,\D g+\frac{2t}{1-t}\sum_{i=1}^N\phi_k\star(\varphi_{k,i}(\D g-\D g_{k,i})+\nabla\varphi_{k,i}\otimes(v_{k,i}-v))\right),
		\\
		J^t_{k,2}(x)&:=J\left(x,\D g+\frac{2t}{1-t}\sum_{i=1}^N\left(\varphi_{k,i}\D g-\phi_k\star(\varphi_{k,i}\D g)\right)\right).
	\end{align*}
	Let us first analyze the error terms, starting with $J_{k,1}^t$. Since $0\leq \varphi_{k,i}\leq 1$ and $\D g_{k,i}\to\D g$ uniformly on $\R^d$, the sequence $\phi_k\star(\varphi_{k,i}(\D g-\D g_{k,i}))$ converges uniformly to $0$ on $\Omega$. For the term $\phi_k\star (\nabla\varphi_{k,i}\otimes (v_{k,i}-v))$ we combine the bound \eqref{eq:grad_control} applied to $a=v_{k,i}-v$ and \eqref{eq:productbound} to deduce that $|\nabla\varphi_{k,i}\otimes (v_{k,i}-v)|$ converges to $0$ in $L^q(B_i)$ (with $q=\infty$ if $p>d-1$) and since ${\rm supp}(\varphi_{k,i})\subset\overline{B_i})$ also in $L^q(\R^d)$. The same convergence then holds for the sequence $\phi_k\star (\nabla\varphi_{k,i}\otimes (v_{k,i}-v))$. Hence, taking into account Assumption \ref{(a3)} we can either use the $q$-growth from above or Lemma \ref{l.interiorpoint} to apply the dominated convergence theorem and find that 
	\begin{equation*}
		\lim_{t\uparrow 1}\lim_{k\to +\infty}\frac{1-t}{2}\int_{\Omega} J^t_{k,1}(x)\dx=\lim_{t\uparrow 1}\frac{1-t}{2}\int_\Omega J(x,\D g)\dx=0.
	\end{equation*}
	For the integral involving $J_{k,2}^t$ the argument is similar: we first simplify the sum by noting that on $\Omega$ we have
	\begin{equation*}
		\sum_{i=1}^N\left(\phi_k\star(\varphi_{k,i}\D g)-\varphi_{k,i}\D g\right)=\phi_k\star \D g-\D g\to 0 \text{ uniformly on }\Omega.
	\end{equation*}
	Thus again Lemma \ref{l.interiorpoint} allows us to apply the dominated convergence theorem and we deduce that
	\begin{equation*}
		\lim_{t\uparrow 1}\lim_{k\to +\infty}\frac{1-t}{2}\int_\Omega J^t_{k,2}(x)\dx=\lim_{t\uparrow 1}\frac{1-t}{2}\int_\Omega J(x,\D g)\dx=0.
	\end{equation*}
	Thus, in order to show \eqref{eq:energyclaim}, it suffices to show that
	\begin{equation}\label{eq:energysupport}
		\lim_{k\to +\infty}\int_\Omega J\left(x,\sum_{i=1}^N \phi_k\star (\varphi_{k,i}(\D v_{k,i}+\D g_{k,i}))\right)\dx= \int_\Omega J(x,\D u)\dx.
	\end{equation}
	First, let us verify the pointwise convergence of the integrand. It holds that
	\begin{align*}
		\left\|\sum_{i=1}^N\varphi_{k,i}(\D v_{k,i}+\D g_{k,i})-\D u\right\|_{L^p(\Omega)}\leq \sum_{i=1}^N\|\D v_{k,i}-\D v\|_{L^p(\Omega)}+\|\D g_{k,i}-\D g\|_{L^p(\Omega)}\to 0,
	\end{align*}
	while outside of $\Omega$ the map $\D v_{k,i}$ vanishes. Since $\D g_{k,i}$ is uniformly bounded, we deduce that
	\begin{equation*}
		\sum_{i=1}^N \phi_k\star (\varphi_{k,i}(\D v_{k,i}+\D g_{k,i}))\to \D u\text{ in }L^p(\Omega)^{m\times d}.
	\end{equation*}
	Hence, in order to show \eqref{eq:energysupport}, due to Vitali's convergence theorem and the a.e. continuity of $J(x,\cdot)$ at $\D u(x)$ established at the beginning of the proof, it suffices to show that the sequence of integrands is bounded by an equi-integrable sequence of functions. It is here where we rely on the full strength of the stability estimate \eqref{eq:stability}. We use it to replace the $x$-variable in the first component of $J$ by $T_{k,i}^{-1}x$ (cf. \eqref{eq:defT_ki}) and to interchange the action of $J$ and the convolution. For technical reasons, we set $J(\cdot,\xi)\equiv 0$ on $\R^d\setminus \Omega$.	
	
	Using a standard estimate for convolutions and that $v_{k,i}$ has compact support in $\Omega$, for any $x\in\R^d$ we have that
	\begin{align*}
		\left|\left(\phi_k\star\sum_{i=1}^N\varphi_{k,i}(\D v_{k,i}+\D g_{k,i})\right)(x) \right|&\leq C\e_k^{-\frac{d}{p}}\left\|\sum_{i=1}^N\varphi_{k,i}(\D v_{k,i}+\D g_{k,i}) \right\|_{L^{p}(B_{\e_k}(x))}
		\\
		&\leq C\e_k^{-\frac{d}{p}}\sum_{i=1}^N\|\D v_{k,i}\|_{L^p(\Omega)}+\|\D g\|_{L^\infty(\R^d)}.
	\end{align*}
	Due to \eqref{eq:definescale}, we deduce that for a suitably large constant $C>0$ we have that
	\begin{equation}\label{eq:uniformbound}
		\left\|\phi_k\star\sum_{i=1}^N\varphi_{k,i}(\D v_{k,i}+\D g_{k,i}) \right\|_{L^{\infty}(\R^d)}\leq C\e_k^{-\frac{d}{p}}\leq C \left|\frac{1}{\rho_k}-1\right|^{-\frac{d}{p}}.
	\end{equation}
	We apply the stability estimate \eqref{eq:stability} for $x\in\Omega$ and the scale $\delta_k=2\text{ diam}(\Omega)\left|\frac{1}{\rho_k}-1\right|$. Due to \eqref{eq:definescale}, we may assume that $\e_k\leq\delta_k$. To simplify notation, set $\omega_{k}(x,\xi)=\left(\text{ess}\inf_{y\in \overline{B_{2\delta_k}(x)}\cap\Omega}J(y,\cdot)\right)^{**}(\xi)$. Then by \eqref{eq:stability}
	\begin{equation*}
		J\left(x,\phi_k\star\sum_{i=1}^N\varphi_{k,i}(\D v_{k,i}+\D g_{k,i})\right)\leq C\omega_k\left(x,\phi_k\star\sum_{i=1}^N\varphi_{k,i}(\D v_{k,i}+\D g_{k,i})\right)+\alpha(x).
	\end{equation*}
	Since $\alpha\in L^1(\Omega)$, we can continue by bounding the remaining RHS term by an equi-integrable sequence of  functions. Recall that $\omega_k$ is convex and lower semicontinuous in the $\xi$-variable, so that it respects Jensen's inequality. Therefore
	\begin{align}\label{eq:w_k_estimate}
		\omega_k\left(x,\phi_k\star\sum_{i=1}^N\varphi_{k,i}(\D v_{k,i}+\D g_{k,i})\right)&\leq\int_{B_{\e_k}(x)}\phi_k(x-y)\omega_k\left(x,\sum_{i=1}^N\varphi_{k,i}(y)(\D v_{k,i}(y)+\D g_{k,i}(y))\right)\dy\nonumber
		\\
		&\leq \sum_{i=1}^N\int_{B_{\e_k}(x)}\phi_k(x-y)\varphi_{k,i}(y)\omega_k(x,\D v_{k,i}(y)+\D g_{k,i}(y))\dy,
	\end{align}
	where the last inequality follows from convexity (recall that $\varphi_{k,0}$ vanishes on a neighborhood of $\overline{\Omega}$, so for $k$ large enough the above weights sum up to $1$ for all $x\in\Omega$). Now we need to distinguish two cases for $y\in B_{\e_k}(x)$: if $T_{k,i}^{-1}y\in\Omega$, then 
	\begin{equation*}
		|T_{k,i}^{-1}y-x|\leq |T_{k,i}^{-1}y-y|+|y-x|<\delta_k +\e_k\leq 2\delta_k,
	\end{equation*}
	so that the elementary inequality $f^{**}\leq f$ and the definition of $\omega_k$ yield that
	\begin{equation}\label{eq:case1}
		\omega_k(x,\D v_{k,i}(y)+\D g_{k,i}(y))\leq J(T_{k,i}^{-1}y,\D v_{k,i}(y)+\D g_{k,i}(y)).
	\end{equation}
	If instead $T_{k,i}^{-1}y\notin \Omega$, then it follows from the definition of $v_{k,i}$ that $\D v_{k,i}(y)=0$. To control the contribution coming from $\D g_{k,i}$, recall that $\D g_{k,i}\to \D g$ uniformly. Together with the uniform continuity of $\D g$, for $k$ large enough and $y\in B_{\e_k}(x)$ we find that 
\begin{align*}
	\w_k(x,\D g_{k,i}(y))\leq \sup_{|\xi|\leq r}\w_k(x,\D g(x)+\xi)\leq \sup_{|\xi|\leq r}J(x,\D g(x)+\xi),
\end{align*}
where $r>0$ is given by Lemma \ref{l.interiorpoint}. Combined with \eqref{eq:case1} we obtain the global bound
\begin{equation*}
	\w_k(x,\D v_{k,i}(y)+\D g_{k,i}(y))\leq J(T_{k,i}^{-1}y,\D v_{k,i}+\D g_{k,i})+\sup_{|\xi|\leq r}J(x,\D g(x)+\xi)
\end{equation*}
for all $x\in\Omega,\;y\in B_{\e_k}(x)$. Inserting this estimate into \eqref{eq:w_k_estimate}, we infer that
	\begin{align*}
		\omega_k\left(x,\phi_k\star\sum_{i=1}^N\varphi_{k,i}(\D v_{k,i}+\D g_{k,i})\right)&\leq \sum_{i=1}^N\int_{B_{\e_k}(x)}\phi_k(x-y)\varphi_{k,i}(y)J(T_{k,i}^{-1}y,\D v_{k,i}+\D g_{k,i})\dy
		\\
		&\quad +\sup_{|\xi|\leq r}J(x,\D g(x)+\xi).
	\end{align*}
	The last RHS function is integrable by Lemma \ref{l.interiorpoint}, so it suffices to show that for each $i\in\{1,\ldots,N\}$ the sequences in the RHS sum are equi-integrable. By the change of variables $T_{k,i}^{-1}y=z$ and the scaling of the mollifier, we obtain that 
	\begin{align*}
		&\quad \int_{B_{\e_k}(x)}\phi_k(x-y)\varphi_{k,i}(y)J\left(T_{k,i}^{-1}y,\D v_{k,i}(y)+\D g_{k,i}(y)\right)\dy
		\\
		&\leq\int_{T_{k,i}^{-1}B_{\e_k}(x)}\phi_k(x-T_{k,i}z)J\left(z,\D v(z)+\D g(z)\right)\,\mathrm{d}z
		\leq C\dashint_{B_{C\e_k}(x)}J(z,\D u(z))\,\mathrm{d}z
		\\
		&=C\frac{1}{(C\e_k)^d|B_1(0)|}\mathds{1}_{B_1(0)}(\cdot/C\e_k)\star J(\cdot,Du(\cdot))(x).
	\end{align*}
	By the properties of approximate identities (cf. \cite[Theorem 8.14]{Fo}), the last sequence converges strongly in $L^1(\Omega)$ (and thus is equi-integrable) since $x\mapsto J(x,\D u(x))\in L^1(\R^d)$ due to the trivial extension by $0$ outside of $\Omega$. We thus conclude the proof.	
\end{proof}

We now prove the first main result for exponents $p\leq d-1$.
\begin{proof}[Proof of Theorem \ref{thm:main}]
	For $u\in W^{1,p}(\Omega)^m$ set $F(u)=\int_\Omega J(x,\D u)\dx$. We show that for any $u\in g+W^{1,p}_0(\Omega)^m$ there exists a sequence $u_n\in g+C_c^{\infty}(\Omega)^m$ such that $u_n\to u$ in $W^{1,p}(\Omega)^m$ and $F(u_n)\to F(u)$. We first reduce the analysis to a simpler situation. Arguing as at the beginning of the proof of Lemma \ref{l.compactsupport}, we may assume without loss of generality that $F(u)<+\infty$. Let $t\in (0,1)$. Then $u_t:=g+t(u-g)\in g+W^{1,p}_0(\Omega)^m$ and by convexity we have 
	\begin{align*}
		F\left(\frac{(1+t)}{2t}(u_t-g)\right)&=F\left(\frac{(1+t)}{2}(u-g)\right)=F\left(\frac{(1+t)}{2}u+\frac{(1-t)}{2}\frac{(1+t)}{(t-1)}g\right)
		\\
		&\leq \frac{(1+t)}{2}F(u)+\frac{(1-t)}{2}F\left(\frac{(1+t)}{(t-1)}g\right)<+\infty,
	\end{align*}
	where we use for the last inequality that by assumption $F(sg)<+\infty$ for all $s\in\R$. Moreover, we have $F(u_t)\leq (1-t)F(g)+tF(u)<+\infty$ and 
	\begin{equation*}
		\lim_{t\uparrow 1}F(u_t)\leq F(u)
	\end{equation*}
	and since $u_t\to u$ in $W^{1,p}(\Omega)^m$ when $t\uparrow 1$ and $\tfrac{1+t}{2t}>1$, a diagonal argument allows us to show the approximation for functions $u$ such that additionally $F(s(u-g))<+\infty$ for some $s>1$. We next apply Lemma \ref{l.compactsupport} with zero boundary datum to find a sequence $v_{n}\in C_c^{\infty}(\Omega)^m$ such that $v_n\to u-g$ in $W^{1,p}(\Omega)^m$ and
	\begin{equation}\label{eq:zerobc}
		\lim_{n\to +\infty}\int_\Omega J(x,s\D v_n)\dx=\int_\Omega J(x,s\D (u-g))\dx<+\infty.
	\end{equation}
	Up to a subsequence (not relabeled), we can also assume $v_n\to u-g$ and $\D v_n\to \D u-\D g$ a.e. in $\Omega$. We show that
	$$
	\lim_{n\to +\infty}\int_\Omega J(x,\D g+\D v_n)\dx=\int_\Omega J(x,\D u)\dx,
	$$
	which proves the assertion. By Fatou's lemma and the non-negativity and lower semicontinuity of $J$, we have
	$$\liminf_{n\to\infty}\int_\Omega J(x,\D g+\D v_n)\dx\geq\int_\Omega J(x,\D u)\dx.$$
	To show the corresponding inequality for the $\limsup$, we first observe that for all $\xi\in\R^{m\times d}$
	\begin{equation}\label{est:tildeWforfatou}
		J(x,\D g+\xi)=J(x,(1-\tfrac{1}{s})\tfrac{s}{s-1}\D g+\tfrac{1}{s}s\xi)\leq J(x,\tfrac{s}{s-1}\D g)+J(x,s\xi).
	\end{equation}
	Hence, the desired inequality follows with help of estimate \eqref{est:tildeWforfatou}, Fatou's lemma and \eqref{eq:zerobc}:
	\begin{align*}
		\limsup_{n\to\infty}\int_\Omega J(x,\D g+\D v_n)\dx\leq&-\liminf_{n\to\infty}\int_\Omega J(x,\tfrac{s}{s-1}\D g)+J(x,s\D v_n)-J(x,\D g+\D v_n)\dx\\
		&+\lim_{n\to\infty}\int_\Omega J(x,\tfrac{s}{s-1}\D g)+J(x,s\D v_n)\dx\\
		\leq&\int_\Omega J(x,\D u)\dx.
	\end{align*}
	Note that here we used that also $\xi\mapsto -J(x,\xi)$ is lower semicontinuous due to the fact that $J(x,\cdot)$ is even continuous, being convex and finite-valued.
\end{proof}
Next, we prove our results on unbounded and locally bounded integrands.
\begin{proof}[Proof of Theorem \ref{thm:2}]
	The result is contained in the statement of Lemma \ref{l.compactsupport}.
\end{proof}
\begin{proof}[Proof of Corollary \ref{c.locallybounded}]
	We can apply Lemma \ref{l.compactsupport} with $g=0$ and then repeat the construction used in the proof of Theorem \ref{thm:main}, where no additional growth assumptions were used except the integrability of $x\mapsto J(x,s\D g(x))$ for every $s\in\R$. This holds for $g\in W^{1,\infty}(\Omega)^m$ since \eqref{eq:stability} implies that $J(\cdot,\xi)\in L^1(\Omega)$ for all $\xi\in\R^{m\times d}$, which combined with convexity yields that $x\mapsto \sup_{|\xi|\leq R}J(x,\xi)\in L^1(\Omega)$ for all $R>0$. For the claimed integrability of $J(\cdot,\xi)$, note that for $\delta\ll 1$ (but fixed) we have
	\begin{equation*}
		J(x,\xi)\leq C\;\text{ess}\inf_{y\in \overline{B_{\delta}(x)}\cap\Omega}J(y,\xi)+\alpha(x)
	\end{equation*}
	for some $\alpha\in L^1(\Omega)$. We prove that the first RHS function is bounded. To this end, assume by contradiction that there exists a sequence $(x_k)_{k\in\N}\subset\Omega$ such that 
	\begin{equation*}
		\text{ess}\inf_{y\in \overline{B_{\delta}(x_k)}\cap\Omega}J(y,\xi)\geq k\qquad\text{ for all }k\in\N.
	\end{equation*}
	Passing to a subsequence (not relabeled), we can assume that $x_k\to x_0$ for some $x_0\in\overline{\Omega}$. But then 
	\begin{equation*}
		\text{ess}\inf_{y\in \overline{B_{\delta/2}(x_0)}\cap\Omega}J(y,\xi)\geq \lim_{ k \to +\infty}\text{ess}\inf_{y\in \overline{B_{\delta}(x_k)}\cap\Omega}J(y,\xi)=+\infty,
	\end{equation*}
	which contradicts the fact that $J$ is finite-valued. 
\end{proof}
We finally treat the non-convex case, where the argument is essentially taken from \cite[Chapter X, Proposition 2.10]{EkTe}. 

\begin{proof}[Proof of Corollary \ref{c.nonconvex}]
Let $u_n\in g+C_c^{\infty}(\Omega)^m$ be the sequence given by either Theorem \ref{thm:main}, Theorem \ref{thm:2} or Corollary \ref{c.locallybounded}. Thus $u_n\to u$ in $W^{1,p}(\Omega)^m$ and $\int_{\Omega}J(x,\D u_n)\dx\to \int_{\Omega}J(x,\D u)\dx$ and after extracting a subsequence (not relabeled) we have $\D u_n\to \D u$ a.e.\ in $\Omega$. In view of Fatou's lemma, it suffices to show
\begin{equation*}
	\limsup_{n\to +\infty}\int_{\Omega}G(x,\D u_n)\dx\leq \int_{\Omega}G(x,\D u)\dx.
\end{equation*}
It is not restrictive to assume that $\int_{\Omega}G(x,\D u)\dx<+\infty$. In this case also $\int_{\Omega}J(x,\D u)\dx<+\infty$. Applying Fatou's lemma to the non-negative integrand $-G(x,\xi)+C J(x,\xi)+\alpha(x)$ we infer that
\begin{align*}
	\int_{\Omega}-G(x,\D u)+CJ(x,\D u)+\alpha(x)\dx&\leq \liminf_{n\to +\infty}\left(\int_{\Omega}-G(x,\D u_n)+CJ(x,\D u_n)+\alpha(x)\dx\right)
	\\
	&\leq -\limsup_{n\to +\infty}\int_{\Omega}G(x,\D u_n)\dx
	\\
	&\quad +\int_{\Omega}CJ(x,\D u)+\alpha(x)\dx.
\end{align*}
Now removing the finite term $\int_{\Omega}CJ(x,\D u)+\alpha(x)\dx$ from both sides, the claim follows by rearrangement.

%
\end{proof}
	
\appendix

\section{The autonomous case on locally strongly star-shaped sets}\label{app:stronglystarshaped}
In this appendix we provide the necessary details to extend our results from open, bounded sets with Lipschitz boundary to the more general class of open, bounded sets that are locally star-shaped. We first recall their definition.

\begin{definition}[(Locally) strongly star-shaped sets] 
	\begin{itemize}
		\item[i)] An open bounded set $\Omega$, which is star-shaped with respect to a point $z\in O$, is called strongly star-shaped if the relative interior of each segment from $z$ to a point in $\partial \Omega$ is contained in $\Omega$.
		\item[ii)] A bounded, open set $\Omega$ is called locally strongly star-shaped if for every $x\in\partial \Omega$ there exists an open neighborhood $O$ of $x$ such that $O\cap \Omega$ is strongly star-shaped.
	\end{itemize}
\end{definition}
It is a classical result that bounded, open sets with Lipschitz boundary are locally strongly star-shaped, but as shown in \cite[Proposition 2.10]{BMT} the boundary of such a set can also have particular cusps. However, as the following elementary lemma shows, the boundary always has measure zero. This fact was used in the proof of Lemma \ref{l.compactsupport}.
\begin{lemma}\label{l.boundarynull}
	Let $\Omega$ be locally strongly star-shaped. Then $|\partial \Omega|=0$.
\end{lemma}
\begin{proof}
	Since $\overline{\Omega}$ is compact, we can write $\Omega$ as a finite union of strongly-star shaped sets. Using the inclusion $\partial (A\cup B)\subset\partial A\cup\partial B$, it thus suffices to show that the boundary of any strongly star-shaped set has zero measure. Hence we assume that $\Omega$ is strongly star-shaped. Moreover, up to a translation it is not restrictive to assume that the center point $z$ is the origin. Consider its radial function $r_\Omega:\mathbb{S}^{d-1}\to (0,+\infty)$ defined by
	\begin{equation*}
		r_\Omega(\nu)= \sup\{\lambda>0:\lambda\nu\in \Omega\}.
	\end{equation*}
	Since $\Omega$ is strongly star-shaped, it follows that $\nu\mapsto r_\Omega(\nu)\nu$ defines a bijection between $\mathbb{S}^{d-1}$ and $\partial \Omega$. Hence the family $(s\partial \Omega)_{s>0}$ is pairwise disjoint, so that only countably many of the sets can have positive measure. But $|s\partial \Omega|=s^d|\partial \Omega|$, which implies that $|\partial \Omega|=0$.
\end{proof}
The following result is \cite[Proposition 2.11]{BMT}.
\begin{lemma}\label{l.stronglystar}
	Let $\Omega$ be a bounded, open set that is locally strongly star-shaped. Then there exist finitely many open sets $G_1,\ldots,G_k\subset\R^d$ that cover $\partial \Omega$ and such that each set $U_i:=G_i\cap \Omega$ is strongly star-shaped with respect to some $z_i\in U_i$. Moreover, there exists $\delta>0$ such that for all $\lambda\in (1-\delta,1)$ it holds that 
	\begin{equation}\label{eq:compactlycontained}
		(z_i+\lambda(\Omega-z_i))\cap U_i\subset\subset \Omega\qquad\text{ for all }i=1,\ldots,k.
	\end{equation}
\end{lemma}
With the above lemma at hand, we can adapt the proof of Lemma \ref{l.compactsupport} as follows: instead of covering the boundary with the cylinders $C_x$, we can use an open set $G_x$ such $x\in G_x$ and such that there exists $z_x\in U_x=G_x\cap\Omega$ satisfying \eqref{eq:compactlycontained}, so in particular for a sequence $\rho_k\uparrow 1$ we have 
\begin{equation}\label{eq:inclusion}
	\dist((z_x+\rho_k(\overline{\Omega}-z_x))\cap U_x,\partial \Omega)=d_{k}>0.
\end{equation}
Note that $(\rho_k)_{k\in\N}$ and $d_{k}$ do not depend on $x$. The point $z_i$ is then replaced by the point $z_{x_i}$ in a finite subcover, while the remaining construction remains the same. In this case the function $v_k$ still has compact support in $\Omega$. Indeed, for $x\in U_{x_i}$ with $\dist(x,\partial \Omega)< d_k$ the equation \eqref{eq:inclusion} yields that $x\notin z_i+\rho_k(\overline{\Omega}-z_i)$ and since $v=0$ outside of $\overline{\Omega}$ we conclude that
\begin{equation}
	v_{k,i}=0\qquad\text{ on }U_{x_i}\cap\{\dist(\cdot,\partial \Omega)<d_k\}.
\end{equation}
Hence the convolution has to be performed at a scale $\e_k\sim d_k $, which in general cannot be compared to $\left|\frac{1}{\rho_k}-1\right|$. However, this was only used for applying the stability estimate \eqref{eq:stability}, which is not needed in the autonomous setting. The remaining arguments did not rely on Lipschitz regularity of the boundary, provided we define $W_0^{1,p}(\Omega)^m$ as the norm-closure of $C_c^{\infty}(\Omega)^m$.

\end{document}